\documentclass[11pt,reqno]{amsart}
\usepackage{tikz}
\usepackage{subfig}
\usepackage{epstopdf}
\usepackage{epsfig}
\usepackage{math}
\graphicspath{{./figures/}}
\usetikzlibrary{shapes,arrows}
\tikzstyle{decision} = [diamond, draw, fill=blue!20, 
    text width=4.5em, text badly centered, node distance=3cm, inner sep=0pt]
\tikzstyle{block} = [rectangle, draw, fill=blue!20, 
    text width=5em, text centered, rounded corners, minimum height=4em]
\tikzstyle{line} = [draw, -latex']
\tikzstyle{cloud} = [draw, ellipse,fill=red!20, node distance=3cm,
    minimum height=2em]
\usetikzlibrary{positioning}
\tikzset{main node/.style={circle,fill=blue!20,draw,minimum size=1cm,inner sep=0pt},  }

\begin{document}
\title[Finite player game]{Equilibrium selection via Optimal transport}
\author[Chow, Li, Lu, Zhou]{Shui-Nee Chow, Wuchen Li, Jun Lu and Haomin Zhou}
\thanks{This work is partially supported by NSF 
Awards DMS\textendash{}1419027,  DMS-1620345, 
and ONR Award N000141310408.}


\maketitle
\begin{abstract}
 We propose a new dynamics for equilibrium selection of finite player discrete strategy games. The dynamics is motivated by optimal transportation, and models individual players' myopicity, greedy and uncertainty when making decisions. The stationary measure of the dynamics provides each pure Nash equilibrium a probability by which it is ranked. For potential games, its dynamical properties are characterized by entropy and Fisher information.
\end{abstract}

\section{Introduction}

Game theory plays a vital role in economics, biology, social network, etc. \cite{IT1, nash1950equilibrium, von2007theory, s2009, s1999}. It models conflict and cooperation between rational decision makers. Each player in a game minimizes his or her own cost function. 
Nash equilibrium (NE) describes a status that no player is
willing to change his or her strategy unilaterally. A fundamental question in game theory is that
if there are multiple pure Nash equilibria, how can one select or rank them?
This problem has been studied previously using various approaches. One classical approach 
\cite{harsanyi1995, harsanyi1988} selects NEs by refining the concept of equilibrium such as payoff dominance and risk dominance principle.
Another class of approaches uses learning dynamics by assuming that the players have bounded knowledge and they need to ``learn'' from what occurred in previous stages of the game and then respond to other players' strategies \cite{Young, Young1}. In these settings irrationalities of individual players are often considered. 
Such examples include
fictitious play, no-regret dynamics, and logit dynamics. To demonstrate the idea, we describe logit dynamics in detail \cite{blume1993}. In logit dynamics, players are assumed to play a game repeatedly. At each time step, one player is selected uniformly at random and its strategy is updated according to a Gibbs-like measure parametrized by a positive number representing the rationality level. 
This process gives rise to a Markov jump process, whose distribution 
converges to a stationary distribution. By vanishing the rationality parameter, the stationary distributions converges to a unique measure, providing crucial information of the stabilities of the NEs and hence giving rise a mechanism for equilibrium selections \cite{logit}. 


On the other hand, for continuous strategy games, equilibrium selection can be done in a rather natural way by stochastic differential equations (SDEs)  and optimal transport theory. Individual players can be modeled to make decisions according to a stochastic process, named best-reply process \cite{de2014}. There 
players change their pure strategies \textit{locally} and simultaneously in a continuous fashion according to the direction that minimizes their own cost function most rapidly. Players' irrationalities are introduced by Brownian motions with a parameter representing irrationality levels. The time evolution of the probability density of the best-reply process is characterized by a Fokker-Planck equation, which is the learning dynamics of the game. For potential games in which all players have the same cost function  named potential, this learning dynamics is the {\em gradient flow} of the free energy in the probability space equipped with Wasserstein metric \cite{am2006,vil2008}. Here the free energy refers to the average of potential plus negative of Shannon-Boltzman entropy, representing the amount of irrationalities or risks taken by the players. This understanding connects the learning dynamics with statistical physics \cite{villani2002review}. Following this connection, if players are purely rational (vanishing the parameter), NEs are stationary points of the players' best-reply process. Thus the invariant measure associated with best-reply dynamics naturally introduces an order of NEs. This ranking method shares many similarities with the one described in \cite{nature2012}, which relates to Conley-Markov matrix. 


Motivated by learning dynamics and continuous best-reply processes, we propose a new learning dynamics for discrete strategy games. A key step is to introduce Markov jump processes in discrete space, inspired by the discrete optimal transport theory recently developed in \cite{chow2012,li-theory}. 
Let $S=S_{1}\times\cdots\times S_{N}$ be the strategy set where $S_{i}$ is the finite discrete strategy set of player $i$ and let $u_i(x)$ be the cost function of player $i$. The best-reply process $X_\beta(t)$ is defined with state space $S$ and the transition probability 
\begin{equation}\label{best_reply}
\begin{split}
&\textrm{Pr}(X_\beta(t+h)=y\mid X_\beta(t)=x)\\
=&\begin{cases}
\sum_{i=1}^{N}(\bar u_i(y)-\bar u_i(x))_+h+o(h) \quad&\textrm{if}~ y\in \cup_{i=1}^N\mathcal{N}_i(x)\ ;\\
1-\sum_{i=1}^N\sum_{y\in \mathcal{N}_i(x)}(\bar u_i(y)-\bar u_i(x))_+h+o(h)\quad &\textrm{if}~ y=x\ ;\\
0\quad& \textrm{otherwise}\ ,
\end{cases}
\end{split}
\end{equation}
where $\mathcal{N}_i(x)$ is the neighborhood of strategy $x$ for player $i$ and $y\in\mathcal{N}_i(x)$ if $y$ and $x$ differ only at $S_i$. $\rho(t,x)$ is the probability density function of $X_{\beta}(t)$ and $\bar{u}_i(x)$ is defined as
\begin{equation*}
\bar u_i(x)=u_i(x)+\beta\log\rho(t,x)\ .
\end{equation*}
Term $\rho(t,x)$ can be described from the perspective of individual players as follows. From the beginning of the repeated play of the game, each player simulate $X_{\beta}(\cdot)$ infinitely many times until time $t$ and $\rho(t,x)$ is the distributions of $X_{\beta}(\cdot)$ at $t$. This interpretation is different from that of fictitious play in that players in fictitious play rely on only one realization of the Markov process while our model depends on infinite many simulations. 

Process $X_{\beta}(t)$ describes players' behaviors with three features. Firstly, $X_{\beta}(t)$ reflects players' myopicity when making decisions. In other words, players make their decisions based solely on the most recent information and within the neighborhood in the strategy set.
Secondly, players select next strategy that decrease their collective cost with highest probability. This is to say players are greedy during the decision-making process. Thirdly, term $\log \rho(t,x)$ introduces randomness in discrete settings. This randomness models players' irrationality due to either making mistakes or taking risks. The latter interpretation allows us to regard $\bar{u}_{i}(x)$ as {\em noisy cost}. 
Intuitively, if a strategy profile has large cost but low probability, its noisy cost will be low and hence encourage players to select the profile.

The density function $\rho(t,x)$ enjoys many appealing mathematical properties. For potential games, it can be regarded as a gradient flow that converges to the minimizer of the free energy. It can be shown that the convergence is exponentially fast and the convergence rate can be accurately characterized by {\em relative Fisher information} \cite{vil2008}, a key concept in statistical physics \cite{Fisher}. In addition, the dissipation of the free energy along this learning dynamics exactly equals the relative Fisher information.

The paper is organized in the following order. In section \ref{Game}, we give a
brief introduction to best-reply dynamics and optimal transport theory in continuous spaces; In section \ref{e}, we describe the mathematical properties of best-reply dynamics via optimal transport defined on discrete strategy games. 
The connection of our model and statistical physics is discussed in section 4. In section \ref{examples}, we illustrate equilibrium selections via the proposed dynamics for some well-known games.

\section{
Equilibrium selection in continuous strategy game}\label{Game}
In this section, we briefly review best-reply dynamics and its connection with optimal transportation theory. 

Consider a game consisting $N$ players $i\in\{1,\cdots,N\}$. Each player $i$
chooses a strategy $x_i$ from a Borel strategy set $S_i$, e.g. $S_i=\mathbb{R}^{n_i}$. Denote $S=S_1\times \cdots\times S_N$. 
Let $x$ be the vector of all players' decision variables: \begin{equation*}
x=(x_1,\cdots, x_N)=(x_i, x_{-i})\in S\ ,\quad \textrm{for any $i=1,\cdots, N$\ ,}\end{equation*}
where we use the notation 
\begin{equation*}
x_{-i}=\{x_1,\cdots, x_{i-1}, x_{i+1},\cdots, x_N\}\ .
\end{equation*}
Each player $i$ has  cost function $u_i: S\rightarrow \mathbb{R}$, where
$u_i(x)$ is a globally Lipchitz continuous function with respect to $x$. The objective of each player $i$ is to minimize the cost function
\begin{equation*}
\min_{x_i\in \mathbb{R}^{n_i}}~u_i(x_{i})=u_i(x_i,x_{-i})\ .
\end{equation*}
A strategy profile $x^*=(x_1^{*},\cdots, x_N^*)$ is a Nash equilibrium (NE) if no player is willing to change his or her current strategy unilaterally
 \begin{equation*}
  u_i(x^*_{i}, x^*_{-i}) \leq u_i(x_{i},x^*_{-i})\quad \textrm{for any $x_i\in S_i$\ , $i=1,\cdots, N$}\ .\label{ne-cont}
 \end{equation*}

It is natural to consider stochastic processes to describe players' decisions-making processes in a game. For each player $i$, instead of
finding $x_{i}^{*}$ satisfying NE directly, he or she plays the game according to a stochastic process $X_{i}(t),~t\in[0,+\infty)$. Here $t$
is an artificial time variable, at which player $i$ selects his or her decision
based on the current strategies of all other players
$X_{j}(t),t\in\{1,\cdots,N\}$. It is important to note that all players make
their decisions simultaneously and without knowing others' decisions. Each
player selects a strategy that decreases his or her own cost most rapidly. To model the uncertainties of decision making, an $N$-dimensional independent Brownian motion is added 
\begin{equation}\label{main}
d X_i= -\nabla_{X_i}u_i (X_i, X_{-i})dt + \sqrt{2\beta} dB_{t}^{i}\ , 
\end{equation}
where $\beta>0$ controls the magnitude of the noise. 
SDE \eqref{main} $X(t)=(X_i(t))_{i=1}^N$ is called the best-reply process. Observe that if a Nash equilibrium exists, it is also the
equilibrium of \eqref{main} with $\beta=0$. It is known that the transition
density function $\rho(t,x)$ of the stochastic process $X(t)$ satisfies the 
Fokker-Planck equation (FPE)
\begin{equation*}
\frac{\partial \rho(t,x)}{\partial t}=\nabla\cdot(\rho(t,x)\big(\nabla_{x_i} u_i(x_i, x_{-i})\big)_{i=1}^N)+\beta \Delta\rho(t,x)\ .
\end{equation*}

In the case that the game is a potential game, i.e. there exists a $C^1$ potential function $\phi~:~\mathbb{R}^N\rightarrow \mathbb{R}$, such that $\nabla_{x_i} u_i(x_i, x_{-i})=\nabla_{x_i}\phi(x)$. The best-reply process \eqref{main} becomes 
\begin{equation*}\label{cont_sde}
dX=-\nabla \phi(X)dt+\sqrt{2\beta} dB_t\ ,
\end{equation*}
which is a perturbed gradient flow, whose density function satisfies
\begin{equation}\label{FP}
\frac{\partial \rho(t,x)}{\partial t}=\nabla\cdot(\rho(t,x)\nabla \phi(x))+\beta \Delta\rho(t,x)\ .
\end{equation} 
The stationary distribution of \eqref{FP} 
is the Gibbs measure given by
\begin{equation*}
\rho^*(x)=\frac{1}{K}e^{-\frac{\phi(x)}{\beta}}\ ,\quad\textrm{where} ~K=\int_{\mathbb{R}^n}e^{-\frac{\phi(x)}{\beta}}dx\ .
\end{equation*}
It's easily seen that the Gibbs measure introduces an order of Nash equilibria in terms of the potential $\phi(x)$. In other words, given two Nash equilibria, the one with larger
density value will be considered more stable. 
One can extend this ranking to general games by studying the invariant measure of \eqref{main}, see \cite{nature2012}.

Equation \eqref{FP} is closely related to optimal transport theory and has a gradient flow interpretation in geometry. This interpretation enables us to derive the new model for discrete strategy games. In short, the optimal transport theory introduces a distance, known as the Wasserstein metric, on the probability density space. Equipped with this metric, the density space forms an infinite dimensional Riemannian manifold. On this manifold, FPE \eqref{FP} is a gradient flow of an informational functional, known as free energy:
\begin{equation}\label{free_energy}
\int_{\mathbb{R}^d}\phi(x)\rho(x)dx+\beta\int_{\mathbb{R}^d}\rho(x)\log\rho(x)dx\ .
\end{equation}
In addition, equation \eqref{FP} can be rewritten as 
\begin{equation*}
\frac{\partial \rho(t,x)}{\partial t}=\nabla\cdot(\rho(t,x)\big(\nabla \phi(x)+\beta \nabla \log \rho(t,x)\big))\ .
\end{equation*}
The term $\nabla(\phi+\log\rho)$ is called the Wasserstein gradient in \cite{am2006}. The corresponding SDE can be understood as 
\begin{equation}\label{NSDE}
dX=-\nabla\big(\phi(X)+\beta\log\rho(t,X)\big)dt\ .
\end{equation}
Notice that process $X$ and its density function are coupled. The term $\log\rho$ corresponds to the Brownian motion in the best-reply SDE. The formulation of \eqref{NSDE} gives the justification of our definition of {\em noise payoff } and motivates the definition of the jump process in discrete strategy games. 




\section{Equilibrium selection in Discrete Strategy set} \label{e}
In this section, we study the time evolution of the probability density function of best-reply process \eqref{best_reply} for discrete strategy games.  We will show that this density function can be viewed as a FPE in discrete settings under optimal transport metric. From this density function, one can calculate the limit distribution of \eqref{best_reply} for ranking NEs. In addition, we will show that for potential games, the FPE is actually a gradient flow.

\subsection{Optimal transport in norm form game}
We first review some notations in game theory \cite{nash1950equilibrium}. Consider a game with $N$ players. Each player $i\in\{1,\cdots, N\}$ chooses a strategy $x_i$ in a discrete strategy set 
\begin{equation*}
S_i=\{1,\cdots, M_i\}
\end{equation*}
where $M_{i}$ is an integer. Denote the joint strategy set 
\begin{equation*}
S=S_1\times\cdots\times S_N\ .
\end{equation*}
Similar to continuous games, each player $i$ has a cost function $u_i:~S\rightarrow \mathbb{R}$, 
\begin{equation*}
u_i(x)=u_i(x_i, x_{-i})\ .
\end{equation*}
If there are only two players ($N=2$), it is customary to write the cost
function in a bi-matrix form $ (A,B^T)$ with $A=(u_1(i, j))_{M_{1}\times
  M_{2}}$, $B^{T}=(u_2(i,j))_{M_{1}\times M_{2}}$
where $(i, j)\in S_1\times S_2$. This form of representation is called normal
form. 

\begin{example}\label{exp}
Two members of a criminal gang are arrested and imprisoned. Each prisoner is given the opportunity either to defect the other by testifying that the other committed the crime, or to cooperate with the other by remaining silent. Their cost matrix is given by
\begin{center}
\begin{tabular}{ r|c|c| }
\multicolumn{1}{r}{}
 &  \multicolumn{1}{c}{player 2 C}
 & \multicolumn{1}{c}{player 2 D} \\
\cline{2-3}
player 1 C & (1, 1) & (3, 0) \\
\cline{2-3}
player 1 D & (0, 3) & (2, 2)\\
\cline{2-3}
\end{tabular}
\end{center}
In this case, the strategy set is $S=\{C, D\}$, where C represents ``Cooperate''
and D represents ``Defect''. The cost function can be represented as $(A, B^T)$, where 
\begin{equation*}
A=\begin{pmatrix}
1 & 3\\
0 & 2
\end{pmatrix},
\quad 
B^T=\begin{pmatrix}
1 & 0\\
3 & 2
\end{pmatrix}.
\end{equation*}
In this example, it is easy to verify that $(D, D)$ is the NE of game.
\end{example}

For a given finite-player game, we construct a corresponding strategy graph as follows.
For each strategy set $S_{i}$, construct a graph $G_i=(S_i, E_{i})$. Two strategies
$x$ and $y$ are connected if player $i$ can switch strategy from $x$ to $y$. 
If the player is free to switch between any two strategies, it makes $G_i$ a complete graph. 
Let $G=(S,E)=G_{1}\Box\cdots\Box G_{N}$ be the Cartesian product
of all the strategy graphs. In other words, $S=S_1\times \cdots\times S_N$ and
$x=(x_{1},\cdots,x_{N})\in S$ and $y=(y_{1},\cdots,y_{N})\in S$ are connected
if their components are different at only one index and these different components
are connected in their component graph. For any $x=(x_1,\cdots, x_N)\in S$, denote
its neighborhood to be $\mathcal{N}(x)$
\begin{equation*}
\mathcal{N}(x)=\{y\in S\mid \textrm{edge}(x,y)\in E \}\ ,
\end{equation*}
and directional neighborhood to be
\begin{equation*}
\mathcal{N}_i(x)=\{(x_1,\cdots, x_{i-1}, y, x_{i+1},\cdots, x_N)\mid y\in S_i,~\textrm{edge}(x_i,y)\in E_{i}\}\ ,
\end{equation*}
for $i=1,\cdots, N$. The definition of $\mathcal{N}_i(x)$ entails that each player selects his or her strategy with other players' strategies fixed. Notice that
\begin{equation*}
  \mathcal{N}(x)=\bigcup_{i=1}^{N}\mathcal{N}_{i}(x)\ .
\end{equation*}
\begin{example}
Consider a two player Prisoner-Dilemma game,  where $S_1=S_2=\{C,
D\}$. The strategy graph is the following.
   \begin{center}
\begin{tikzpicture}[-,shorten >=1pt,auto,node distance=3cm,
        thick,main node/.style={circle,fill=blue!20,draw,minimum size=1cm,inner sep=0pt]}]
   \node[main node] (1) {$C, C$};
    \node[main node] (2) [right =2cm]  {$C, D$};
    \node[main node] (3) [below =2cm]  {$D, C$};
    \node[main node] (4) [right =1.5 cm of 3]  {$D, D$};

    \path[draw,thick]
    (1) edge node {} (2)
    (2) edge node {} (4)
    (1) edge node {} (3);
\path[draw, thick]
    (4) edge node {} (3);
\end{tikzpicture}
\end{center}
\end{example}
We now introduce an optimal transport distance on the probability space of the strategy graph. The probability space (i.e. a simplex) on all strategies is given by:
\begin{equation*}
\mathcal{P}(S)=\{( \rho(x))_{x\in S}\in \mathbb{R}^{|S|}\mid \quad\sum_{x\in S} \rho(x)=1\ ,\quad  \rho(x)\geq 0\ , \quad \textrm{for any $x\in S$}\}\ ,
\end{equation*}
where $ \rho(x)$ is the  probability at each vertex $x$, and $|S|$ is total number of strategies. Denote the interior of $\mathcal{P}(S)$ by
 $\mathcal{P}_o(S)$. 
 
Given any function $\Phi\colon S\to \mathbb{R}$ on strategy set $S$, define
$\nabla\Phi\colon S\times S\to \mathbb{R}$ as
\begin{equation*}
\nabla\Phi(x,y)=\begin{cases} \Phi(x)-\Phi(y)\quad &\textrm{if $(x,y)\in E$\ ;}\\
0\quad &\textrm{otherwise}\ .
\end{cases}
\end{equation*} 
Let $m\colon S\times S\to \mathbb{R}$ be an anti-symmetric flux function such that $m(x,y) = -m(y,x)$. The divergence of $m$, denoted as $\textrm{div}(m)\in \mathbb{R}^{|S|}$, is defined by \begin{equation*}
\textrm{div}(m)(x) = -\sum_{y\in \mathcal{N}(x)}m(x,y)\ .
\end{equation*}
For the purpose of defining our distance function, we will use a particular flux function
\begin{equation*}
m(x,y)=\rho \nabla\Phi:=g(x,y,\rho)\nabla\Phi(x,y)\ ,
\end{equation*}
where $g(x,y, \rho)$ represents the discrete probability (weight) on $\mathrm{edge}(x,y)$ and satisfies
 \begin{equation}\label{gxy}
g(x,y,\rho)=g(y,x, \rho)\ ,\quad \min\{\rho(x), \rho(y)\}\leq g(x,y,\rho)\leq \max\{\rho(x), \rho(y)\}\ .
\end{equation}
A particular choice of $g(x,y,\rho)$ is of up-wind scheme type, whose explicit formulation will be given shortly.

We can now define the discrete inner product on $\mathcal{P}_o(S)$:
\begin{equation*}
(\nabla\Phi,\nabla\Phi )_\rho:=\frac{1}{2}\sum_{(x,y)\in E} (\Phi(x)-\Phi(y))^2g(x,y,\rho)\ ,
\end{equation*}
which induces the following distance on $\mathcal{P}_o(S)$. 
\begin{definition}
Given two discrete probability function $\rho^0$, $ \rho^1\in\mathcal{P}_o(S)$, define the optimal transport metric function $\mathcal{W}$:
\begin{equation*}\label{metric}
\mathcal{W}(\rho^0,\rho^1)^2=\inf \{\int_0^1(\nabla\Phi,\nabla\Phi)_\rho dt~:~ \frac{d\rho}{dt}+\mathrm{div}(\rho\nabla\Phi)=0\ ,~\rho(0)=\rho^0,~\rho(1)=\rho^1\}\ .
\end{equation*}
\end{definition}
$(\mathcal{P}_o(S), \mathcal{W})$ is a well defined finite dimensional Riemannian manifold \cite{chow2012, maas2011gradient}, which enables us to define the gradient flow (FPE) in $\mathcal{P}_o(S)$.

\subsection{FPEs for potential games}
We first derive the FPE for discrete potential games. Here a potential game means that, {\em there exists a potential function $\phi:~S\rightarrow \mathbb{R}$,
such that
\begin{equation*}
\phi(x)-\phi(y)=u_i(x)-u_i(y)\ , \quad\textrm{for any $x,y\in S_i$ and $i=1,\cdots, N$\ .}
\end{equation*}}
As in the continuous case \eqref{free_energy}, our objective functional in $\mathcal{P}(S)$
is the discrete free energy
\begin{equation*}
\sum_{x\in S}\phi(x)\rho(x)+\beta \sum_{x\in S}\rho(x)\log\rho(x)\ ,
\end{equation*}
where the first term is average of potential and the second one is the linear entropy modeling risk-taking.

Using this objective functional, we construct the metric $\mathcal{W}$ with a upwind type $g(x,y,\rho)$ satisfying \eqref{gxy}:
$$g(x,y,\rho)= \begin{cases}
\rho(x) &\textrm{if $\phi(x)+\beta\log \rho(x)>\phi(y)+\beta\log\rho(y)$;}\\
\rho(y)& \textrm{if $\phi(x)+\beta\log \rho(x)<\phi(y)+\beta\log\rho(y)$;}\\
\frac{\rho(x)+\rho(y)}{2}& \textrm{if $\phi(x)+\beta\log \rho(x)=\phi(y)+\beta\log\rho(y)$.}\\
\end{cases}
$$ 
\begin{theorem}[Gradient flow]\label{th12}
Given a potential game with strategy graph $G=(S, E)$, potential $\phi(x) $ and constant $\beta\geq 0$.
\begin{itemize}
\item[(i)]
The gradient flow of 
\begin{equation*}
\sum_{x\in S}\phi(x)\rho(x)+\beta\sum_{x\in S}\rho(x)\log\rho(x)\ ,
\end{equation*}
on the metric space $(\mathcal{P}_o(S), \mathcal{W})$ is the FPE 
\begin{equation*}\label{FPd}
\begin{split}
\frac{ d\rho(t,x)}{dt}&=\sum_{y\in \mathcal{N}(x)} \rho(t,y)[ \phi(y)- \phi(x)+\beta(\log\rho(t,y)-\log\rho(t,x))]_+\\
&-\sum_{y\in \mathcal{N}(x)}\rho(t,x)[ \phi(x)- \phi(y)+\beta(\log\rho(t,x)-\log\rho(t,y))]_+\ .\\
\end{split}
\end{equation*}
\item[(ii)] For $\beta>0$, Gibbs measure 
\begin{equation*}\label{gibd}
\rho^*(x)=\frac{1}{K}e^{-\frac{\phi(x)}{\beta}}\ ,\quad\textrm{where}\quad K=\sum_{x\in S} e^{-\frac{\phi(x)}{\beta}}\ ,
\end{equation*}
is the unique stationary measure of ODE \eqref{FPd}. 
\item[(iii)]  For any given initial condition $\rho^0\in\mathcal{P}_o(S)$, there
  exists a unique solution $\rho(t): [0,\infty)\rightarrow \mathcal{P}_o(S)$ to
  equation \eqref{FPd}.
\end{itemize}
\end{theorem}
The proof follows \cite{chow2012, li-theory}, so omitted here.

\subsection{FPE for discrete strategy games}\label{general_FPE}
For general games, as in the continuous case, the
FPE, the time evolution of probability function of $X_\beta(t)$ in \eqref{best_reply}, can't be interpreted as gradient flows for some
functional. To establish FPEs 
discrete settings, we observe that in \eqref{FPd}, if the underlying graph corresponds to the
Cartesian grid partition, \eqref{FPd} is the
numerical discretization of the continuous FPE using
upwind scheme, see \cite{li-computation}. This motivates us to define the following FPE.
\begin{definition}
For a general game with strategy graph $G=(S,E)$ with cost functionals $u_{i}(x)$
for $i\in{1,\cdots,N}$, define its FPE to be
\begin{equation}\label{flow}
\begin{split}
\frac{d \rho(t,x)}{dt}=&\sum_{i=1}^N\sum_{y\in \mathcal{N}_i(x)}[u_i(y)- u_i(x)+\beta(\log\rho(t,y)-\log\rho(t,x))]_+\rho(t,y)\\
-&\sum_{i=1}^N\sum_{y\in \mathcal{N}_i(x)}[u_i(x)-u_i(y)+\beta(\log\rho(t,x)-\log\rho(t,y))]_+\rho(t,x)\ .
\end{split}
\end{equation}
\end{definition}
Notice that $\cup_{i=1}^{N}N_{i}(x)=\mathcal{N}(x)$. So when the general game is a potential game, the above FPE coincides with \eqref{FPd}.
Our main result for general games is the following theorem. 
\begin{theorem}[General flow]\label{th13}
Given a $N$-player game with strategy graph $G=(S,E)$, cost functional $u_i$, $i=1,\cdots, N$ and a constant $\beta\geq 0$.
\begin{enumerate}
\item[(i)] For all $\beta> 0$ and any initial condition $\rho(0)\in \mathcal{P}_o(S)$, there exists a unique solution 
\begin{equation*}
\rho(t): [0, \infty)\rightarrow \mathcal{P}_o(S)
\end{equation*} 
 of \eqref{flow}.
\item[(ii)] Given any initial condition $\rho_{0}(t)$, denote $\rho^\beta(t)$ the solutions of \eqref{flow} with varying
  $\beta$'s. Then for any fixed time $T\in (0,
  +\infty)$
\begin{equation*}
\lim_{\beta\rightarrow 0}\rho^\beta(t)=\rho^0(t)\ ,\quad t\in[0, T]\ .
\end{equation*}
\item[(iii)] Assume there are $k$ distinct pure Nash equilibria $x^1, \cdots, x^k\in S$. Let $\rho^*(x)$ be a measure such that $$\textrm{Support of}~\rho^*(x)\subset \{x^1,\cdots, x^k\}\ ,$$
then $\rho^*(x)$ is the stationary solution of \eqref{flow} with $\beta=0$.
\end{enumerate}
\end{theorem}
\begin{proof}
(i) is a slight modification of results in \cite{li-computation}. (ii) Let's denote ODE \eqref{flow} for $\beta>0$ as a matrix form 
\begin{equation*}
\frac{d\rho^{\beta}(t)}{dt}=Q(\rho,\beta)\rho^\beta(t)\ .
\end{equation*}
We observe that if $\beta=0$, $Q(\rho,\beta)=Q$ is a constant matrix. 
By the similar reason in proving Theorem \ref{th12}, we know that for any initial condition $\rho^0$, there exists a compact set $B(\rho^0)\subset \mathcal{P}_o(S)$, such that 
$\rho^\beta(t)\in B(\rho^0)$ for any $\beta$. Hence there exists a constant $M>0$, such that
\begin{equation*}
\|(Q(\rho,\beta)-Q)\rho^\beta(t)\|\leq M\beta\ ,
\end{equation*}
where $\|\cdot\|$ is the 2-norm.
In other words, the difference of the ODE \eqref{flow}'s solution at $\beta>0$ and $\beta=0$ is
\begin{equation*}
\begin{split}
\frac{d(\rho^\beta(t)-\rho^0(t))}{dt}=&Q(\rho^\beta, \beta)\rho^\beta-Q\rho^0\\
=&Q(\rho^\beta-\rho^0)+(Q(\rho^\beta, \beta)-Q)\rho^\beta\ .
\end{split}
\end{equation*}
Hence 
  \begin{equation*}
  \begin{split}\frac{d\|\rho^{\beta}(t)-\rho^0(t)\|}{dt}\leq &\|Q(\rho^{\beta}(t)-\rho^0(t))\|+\|(Q(\rho^\beta, \beta)-Q)\rho^\beta\|\\
\leq &\|Q\|\|\rho^{\beta}-\rho^0\|+\beta M\ .\\
\end{split}
\end{equation*}
By Gronwall's inequality, for $t\in [0, T]$, we have
\begin{equation*}
\|\rho^{\beta}(t)-\rho^0(t)\|\leq \beta M e^{\|Q\|T}\ ,
\end{equation*}
which finishes the proof.

We now prove (iii). Denote $\mathcal{E}=\{x^1,\cdots, x^k\}$, 
then $\textrm{Support of}~\rho^*(x)\subset \mathcal{E}$ implies
\begin{equation}\label{def}
\rho^*(x)=\begin{cases}
0 & \textrm{if $x\not\in \mathcal{E}$;}\\
\geq 0 & \textrm{if $x\in \mathcal{E}$.}
\end{cases}
\end{equation}
Since $x\in \mathcal{E}$ is a NE, $u_i(y)\geq u_i(x)$ when $y\in \mathcal{N}_i(x)$, for any $i=1,\cdots, d$. For $x\in \mathcal{E}$, we substitute $\rho^*(x)$ into the R.H.S.  \eqref{flow},  which forms
\begin{equation*}
\begin{split}
&\sum_{i=1}^N\sum_{y\in \mathcal{N}_i(x)}[u_i(y)- u_i(x)]_+\rho^*(y)-\sum_{i=1}^N\sum_{y\in \mathcal{N}_i(x)}[u_i(x)-u_i(y)]_+\rho^*(x)\\
=& \sum_{i=1}^N\sum_{y\in \mathcal{N}_i(x)}[u_i(y)- u_i(x)]\rho^*(y)-0\\
=&0\ ,
\end{split}
\end{equation*}
 where the last equality is from the following facts in two cases. (i) If $y\not\in \mathcal{E}$, $\rho^*(y)=0$ from \eqref{def}. (ii) if $y\in \mathcal{E}$, 
  $u_i(y)\geq u_i(x)$, then $u_i(y)-u_i(x)=0$. 
Similarly, we can show the case when $x\not\in \mathcal{E}$. 
\end{proof}

\subsection{Nash equilibria selection}
FPE gives the stationary distributions (equilibrium) for the dynamics. It allows us to rank different equilibria by comparing the probabilities.

 For potential games, the stationary distribution is the Gibbs measure,
which provides the same ranking as that given by simply comparing
potentials. Denote $x^{1}, \cdots, x^{k}\in S$ as distinct NEs. A natural order is as follows:
\begin{equation}\label{rank2}
x^{1}\prec x^{2}\cdots\prec x^{k}, \quad \textrm{if} \quad \rho^*(x^{1})< \cdots < \rho^*(x^{k})\ .
\end{equation}
 Here $x\prec y$ is to say that the strategy $y$ is better(more stable) than strategy $x$.
The above definition is equivalent to look at $\phi(x^1)< \cdots < \phi(x^k)$, since $\rho^*(x)=\frac{1}{K}e^{-\frac{\phi(x)}{\beta}}$.

For non-potential games, although there is no potentials, the stationary solution of FPE $\rho^*(t)$ still provides a way of ranking equilibria. We call it the {\em transport order of NEs}. 
\begin{definition}[Transport order of NEs]
Assume $\rho^*(x)=\lim_{\beta\rightarrow 0}\lim_{t\rightarrow \infty} \rho(t,x)$ exits, where $\rho(t,x)$ is the solution of \eqref{flow} with any initial measure $\rho^0\in \mathcal{P}_o(S)$.
We define the order of NE by \begin{equation}\label{rank2}
x^{1}\prec x^{2}\cdots\prec x^{k}, \quad \textrm{if} \quad \rho^*(x^{1})< \cdots < \rho^*(x^{k})\ .
\end{equation}

\end{definition}

In Section \ref{examples}, we will give several examples to illustrate this selection method. 

\section{Entropy dissipation}
In this section, we illustrate the connection between our Markov process and statistical physics, named the discrete H theory. We will mainly focus on potential games. We borrow two ``discrete'' physical functionals to measure the closeness between two discrete measures, $\rho$ and $\rho^\infty(x)=\frac{1}{K}e^{-\frac{\phi(x)}{\beta}}$. One is the discrete relative entropy (H)
\begin{equation*}
\mathcal{H}(\rho|\rho^\infty):=\sum_{x\in S} \rho(x) \log \frac{\rho(x)}{\rho^\infty(x)}\ .
\end{equation*}
The other is the discrete relative Fisher information (I)
\begin{equation*}
\mathcal{I}(\rho|\rho^\infty):=\sum_{(x,y)\in E}(\log\frac{\rho(x)}{\rho^\infty(x)}-\log\frac{\rho(y)}{\rho^\infty(y)})_+^2\rho(x)\ .
\end{equation*}

The H theory states that the relative entropy decreases along player's decision process. The following theorem can be viewed as discrete H theorem for finite player games.
\begin{theorem}[Discrete H theorem]
Suppose $\rho(t)$ is the transition probability of $X_\beta(t)$ in potential games.
Then the relative entropy decreases  
\begin{equation*}
\frac{d}{dt}\mathcal{H}(\rho(t)|\rho^\infty)<0\ .
\end{equation*}
And the dissipation of relative entropy is $\beta$ times relative Fisher information  
\begin{equation}\label{Fisher}
\frac{d}{dt}\mathcal{H}(\rho(t)|\rho^\infty)=-\beta\mathcal{I}(\rho(t)|\rho^\infty)\ .
\end{equation}
\end{theorem}

\begin{proof}
Since $\mathcal{I}(\rho|\rho^\infty)\geq 0$ and equality is achieved if and only if $\rho=\rho^\infty$, we only need to prove \eqref{Fisher}. 
Substituting $\rho^\infty(x)=\frac{1}{K}e^{-\frac{\phi(x)}{\beta}}$ into the relative entropy, we observe
\begin{equation*}
\begin{split}
\mathcal{H}(\rho|\rho^\infty)=&\sum_{x\in S} \rho(x) \log \frac{\rho(x)}{\rho^\infty(x)} \\
 =&  \sum_{x\in S} \rho(x) \log \rho(x)-  \sum_{x\in S}\rho(x)\log \rho^\infty(x)\\ 
=&     \sum_{x\in S} \rho(x) \log \rho(x)+\frac{1}{\beta} \sum_{x\in S}\rho(x)\phi(x)+\log K\sum_{x\in S}\rho(x)    \\
=&\frac{1}{\beta} (\beta\sum_{x\in S} \rho(x) \log \rho(x)+\sum_{x\in S}\rho(x)\phi(x))+\log K\ .
\end{split}
\end{equation*}

From the explicit formulation of FPE \eqref{FPd}, we have
\begin{equation*}
\begin{split}
\frac{d}{dt}\mathcal{H}(\rho(t)|\rho^\infty)=&\frac{1}{\beta}\frac{d}{dt}\{\beta\sum_{x\in S} \rho(t,x) \log \rho(t,x)+\sum_{x\in S}\rho(t,x)\phi(t,x)\}\\
=&-\frac{1}{\beta}\sum_{(x,y)\in E}(\phi(x)+\beta\log\rho(t,x)-\phi(y)-\beta\log\rho(t,y) )_+^2\rho(t,x) \\
=&-\frac{1}{\beta}\cdot \beta^2\cdot \sum_{(x,y)\in E}(\log\frac{\rho(t,x)}{\rho^\infty(x)}-\log\frac{\rho(t,y)}{\rho^\infty(y)} )_+^2\rho(t,x)\\
=&-\beta\cdot \mathcal{I}(\rho(t)|\rho^\infty)\leq 0\ ,
\end{split}
\end{equation*}
which finishes the proof.
\end{proof}

Besides the discrete H theorem, there is a deep connection between FPE \eqref{FPd} and statistical physics from the mathematical viewpoint. 
This connection is known as entropy dissipation, i.e. the relative entropy decreases to zero exponentially. We show similar results for the proposed model.
\begin{theorem}[Entropy dissipation]\label{th14}
Given a potential game with $\beta>0$, $\rho^0\in \mathcal{P}_o(S)$, there exists a constant 
$C=C(\rho^0,G)>0$ such that 
\begin{equation}\label{exp}
\mathcal{H}( \rho(t)| \rho^{\infty})\leq e^{-Ct}\mathcal{H}(\rho^0|\rho^\infty)\ .
\end{equation}
\end{theorem}
The proof of Theorem \ref{th14} is presented in \cite{li-theory}. 
\section{Examples}\label{examples}
We give several examples to illustrate the model.  

\noindent{\em Example 1:}
Consider a two-player Prisoner Dilemma $(A, B^T)$ game with cost matrix 
\begin{equation*}
A=B=\begin{pmatrix}1 & 3 \\ 0 & 2\end{pmatrix}\ .
\end{equation*}
Here the strategy set is $S=\{(C,C), (C,D), (D,C), (D, D)\}$. This particular game is a potential game, with 
\begin{equation*}
\phi(x)=-(u_1(x)+u_2(x))\ ,\quad \textrm{where $x\in S$\ .}
\end{equation*}
The strategy graph is $G=K_2\Box K_2$. 
   \begin{center}
\begin{tikzpicture}[-,shorten >=1pt,auto,node distance=3cm,
        thick,main node/.style={circle,fill=blue!20,draw,minimum size=1cm,inner sep=0pt]}]
   \node[main node] (1) {$C, C$};
    \node[main node] (2) [right =2cm]  {$C, D$};
    \node[main node] (3) [below =2cm]  {$D, C$};
    \node[main node] (4) [right =1.5 cm of 3]  {$D, D$};

    \path[draw,thick]
    (1) edge node {} (2)
    (2) edge node {} (4)
    (1) edge node {} (3);
\path[draw, thick]
    (4) edge node {} (3);

\end{tikzpicture}
\end{center}
To simplify notation, we denote the transition probability function as
\begin{equation*}
\rho(t)=(\rho_{CC}(t), \rho_{CD}(t), \rho_{DC}(t), \rho_{DD}(t))^T,
\end{equation*}
which satisfies FPE \eqref{flow}. By numerically solving \eqref{flow} for $\rho^*=\lim_{\beta\rightarrow 0}\lim_{t\rightarrow \infty}\rho(t)$, we find a unique invariant measure $\rho^*$ for any initial condition $\rho(0)$, which is demonstrated in Figure \ref{fig1}. 
\begin{figure}[H]
  \centering
{\includegraphics[scale=0.3]{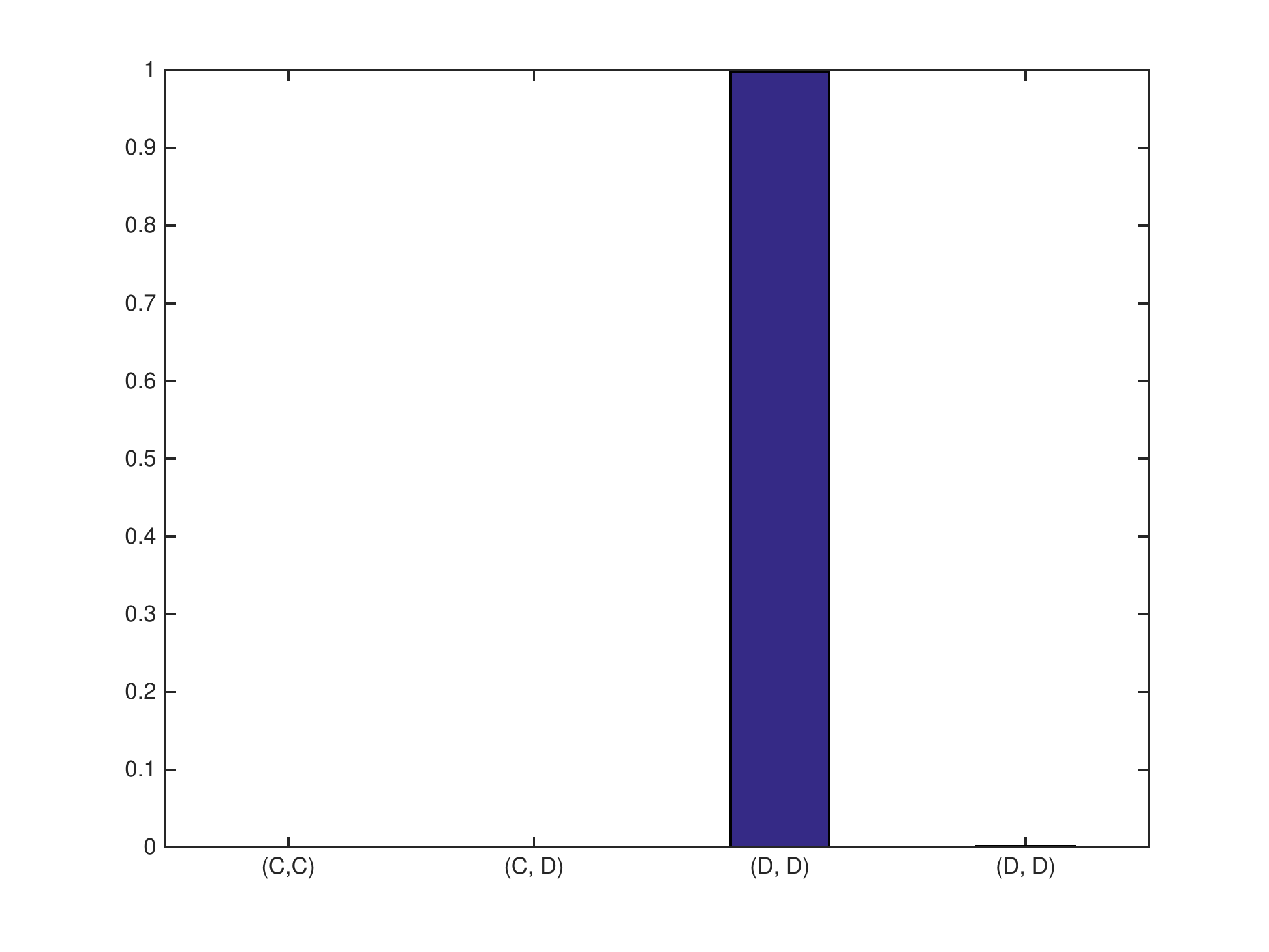}}
\caption{The invariant measure $\rho^*$ for Prisoner Dilemma.}
\label{fig1}
\end{figure}
Indeed, we know that $\rho^*$ is a Gibbs measure and $(D,D)$ is the unique
Nash equilibrium.

\noindent{\em Example 2}: Consider an asymmetric game $(A, B^T)$, i.e. $A\neq B$. This means 
players' cost depend on their own identity. Let
$A=\begin{pmatrix}
1 & 2\\ 2 &1 
\end{pmatrix}$ and 
$B=\begin{pmatrix}
1 & 3\\ 2 &1 
\end{pmatrix}$.
This game is not a potential game.  Again the strategy graph is $G=K_2\Box K_2$.   \begin{center}
\begin{tikzpicture}[-,shorten >=1pt,auto,node distance=3cm,
        thick,main node/.style={circle,fill=blue!20,draw,minimum size=1cm,inner sep=0pt]}]
   \node[main node] (1) {$C, C$};
    \node[main node] (2) [right =2cm]  {$C, D$};
    \node[main node] (3) [below =2cm]  {$D, C$};
    \node[main node] (4) [right =1.5 cm of 3]  {$D, D$};

    \path[draw,thick]
    (1) edge node {} (2)
    (2) edge node {} (4)
    (1) edge node {} (3);
\path[draw, thick]
    (4) edge node {} (3);

\end{tikzpicture}
\end{center}
By solving \eqref{flow} for $\rho^*=\lim_{\beta\rightarrow 0}\lim_{t\rightarrow \infty}\rho(t)$, we obtain a unique $\rho^*$ for any initial condition $\rho(0)$, which
is shown in Figure \ref{fig2}. 
\begin{figure}[H]
  \centering
{\includegraphics[scale=0.3]{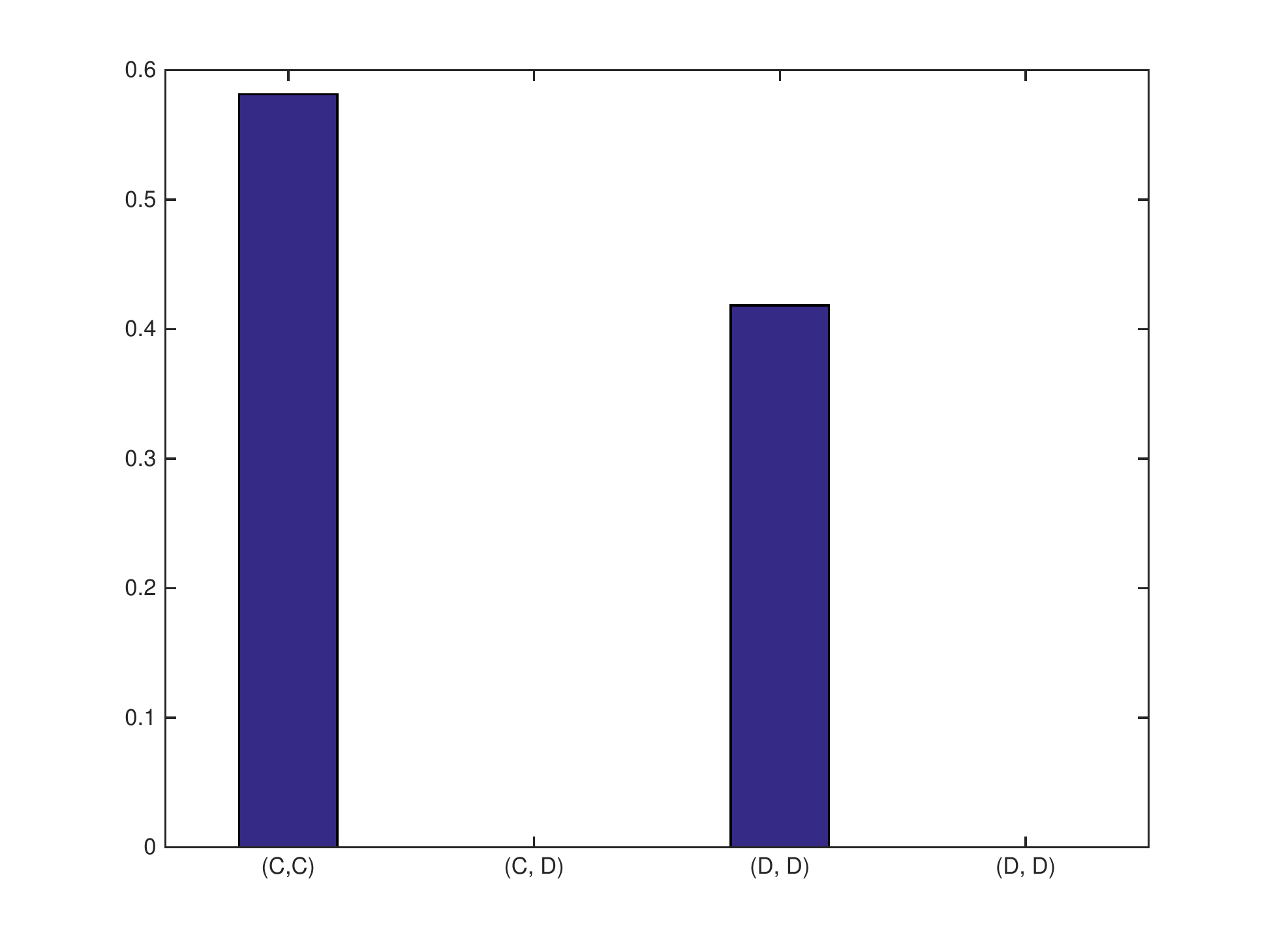}}
\caption{The invariant measure $\rho^*$ for asymmetric game.}
\label{fig2}
\end{figure}
As we can see, $\rho^*$ only supports at $(C,C)$ and $(D,D)$, both of which are
Nash equilibria of the game. Moreover, $\rho^*_{CC}$ is larger than
$\rho^*_{DD}$, which implies that $(C,C)$ is more ``stable'' than $(D,D)$. This
is intuitive because player 2 is more willing to change his/her status from $(C,D)$ to $(C,C)$ than player 1 to move the status $(D,C)$ to $(D, D)$, since player 2's cost changes more rapidly than the one of player 1: $u_2(C,D)-u_2(C,C)=2>1=u_1(D,C)-u_1(D, D)$. 

\noindent{\em Example 3}: Consider a Rock-Scissors-Paper game $(A, B^T)$ with the strategy sets $S_1=S_2=\{r, s, p\}$ and the cost matrix
\begin{equation*}
A=B=\begin{pmatrix}
0&-1&1\\
1&0&-1\\
-1&1&0\\
\end{pmatrix}\ .
\end{equation*}
The strategy graph is $G=K_3\Box K_3$:
\begin{center}
\begin{tikzpicture}[-,shorten >=1pt,auto,node distance=3cm,
        thick,main node/.style={circle,fill=blue!20,draw,minimum size=1cm,inner sep=0pt]}]
   \node[main node] (1) {$r, r$};
   \node[main node] (2) [below left=1.5cm and 0.6 cm of 1]   {$r, s$};  
   \node[main node] (3) [below right= 1.5cm and 0.6 cm of 1]  {$r, p$};
   
  \node[main node] (4) [below left =1 cm and 1 cm of 2]  {$s, r$};
  \node[main node] (5) [below left =1.5 cm  and 0.6 cm of 4]  {$s, s$};
  \node[main node] (6) [below right =1.5 cm and 0.6 cm of 4]  {$s, p$};

  \node[main node] (7) [below right =1 cm and 1 cm of 3]  {$p, r$};
  \node[main node] (8) [below left =1.5 cm and 0.6 cm of 7]  {$p, s$};
  \node[main node] (9) [below right =1.5 cm and 0.6 cm of 7]  {$p, p$};

    \path[draw,thick]
    (1) edge node {} (2)
   (2) edge node {} (3)
   (1) edge node{}(3)
 
  (4) edge node {} (5)
  (5) edge node {} (6)
  (4) edge node {} (6)
  
  (7) edge node {} (8)
  (8) edge node {} (9)
  (7) edge node {} (9)
 
  (1) edge [bend right] node {} (4)
  (2) edge [bend right ]node {} (5)
  (3) edge [bend right] node {} (6)

  (1) edge [bend left] node {} (7)
  (2) edge [bend left] node {} (8)
  (3) edge [bend left] node {} (9)

  (4) edge [bend right] node {} (7)
  (5) edge [bend right] node {} (8)
  (6) edge [bend right] node {} (9);
\end{tikzpicture}
\end{center}

Again, we obtain a unique invariant $\rho^*$ for any initial condition $\rho(0)$ in Figure \ref{fig3}. 
\begin{figure}[H]
  \centering
{\includegraphics[scale=0.3]{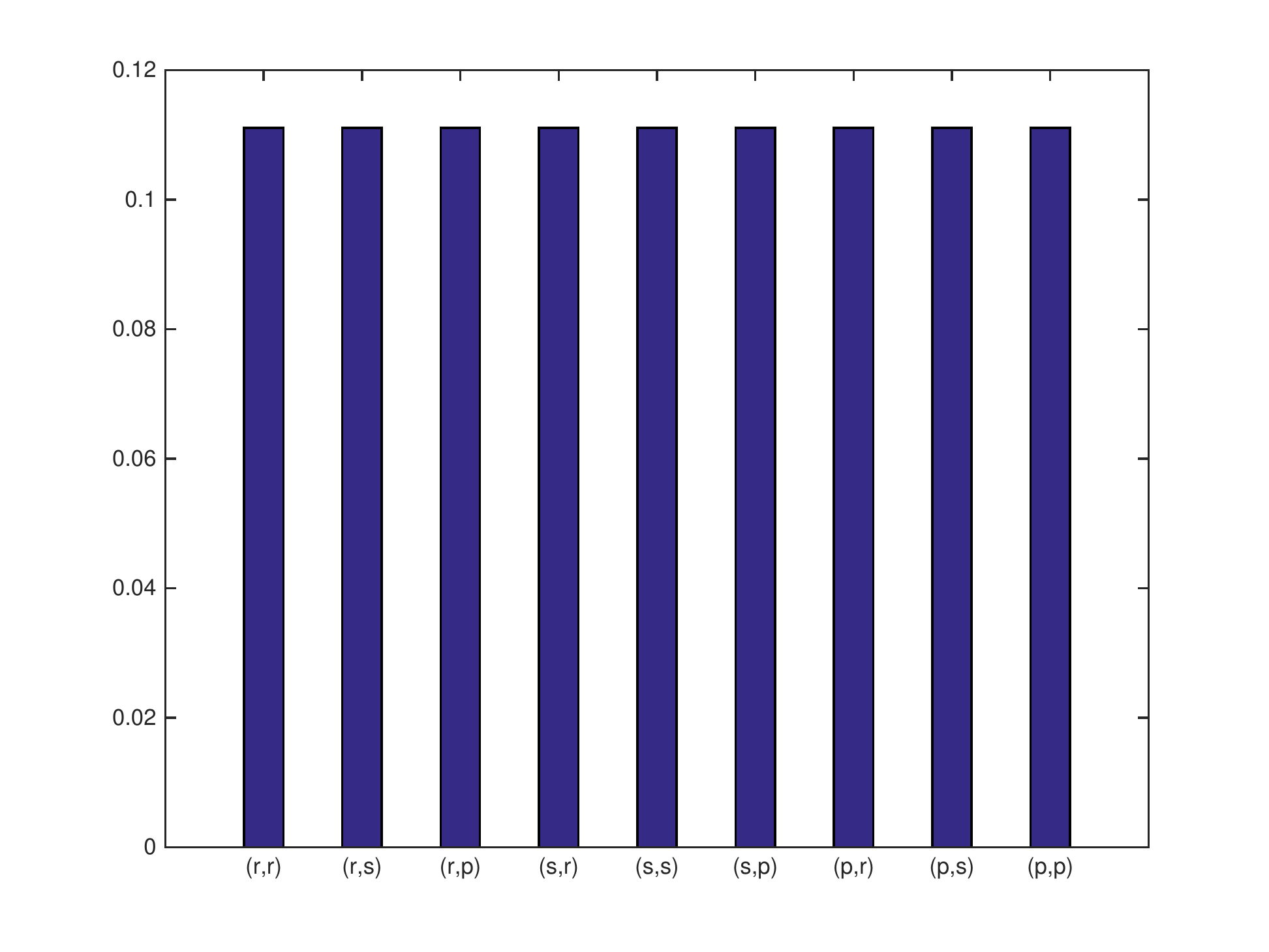}}
\caption{The invariant measure $\rho^*$ for Rock-Scissors-Paper.}
\label{fig3}
\end{figure}
From the figure, we find that the invariant measure $\rho^*$ is a uniform measure. We conclude that, although each player chooses his/her own strategy depending on each others, at the final time, they will arrive at a state that players select strategies uniformly and independently. 

\noindent{\em Example 4}. We consider the same Rock-Scissors-Paper game with
constraints, in order to illustrate the effect of the structure of the strategy graph on stationary joint probability $\rho^*$.
Here the constraint is that player 1 is not allowed to play Scissors following Rock and vice versa. There is no restriction on player 2. The corresponding strategy graph $S_{1}$ is in Figure \ref{rspincomplete} while the strategy graph $S_2$ is a complete graph. We consider $S_1\Box S_2$ for FPE \eqref{a1} and solve for the invariant measure $\rho^*$.
\begin{figure}[H]
\begin{tikzpicture}[-,shorten >=1pt,auto,node distance=3cm,
        thick,main node/.style={circle,fill=blue!20,draw,minimum size=1cm,inner sep=0pt]}]
   \node[main node] (1) {$r$};
   \node[main node] (2) [below left=1.5cm and 0.6 cm of 1]   {$s$};  
   \node[main node] (3) [below right= 1.5cm and 0.6 cm of 1]  {$p$};
   
    \path[draw,thick]
    (1) edge node {} (3)
   (2) edge node {} (3);
\end{tikzpicture}
\caption{Player 1's strategy graph}
\label{rspincomplete}
\end{figure}
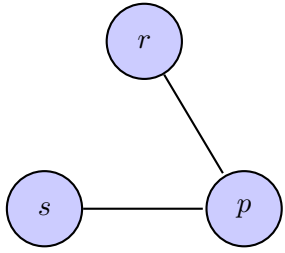
\begin{figure}[H]
  \centering
{\includegraphics[scale=0.4]{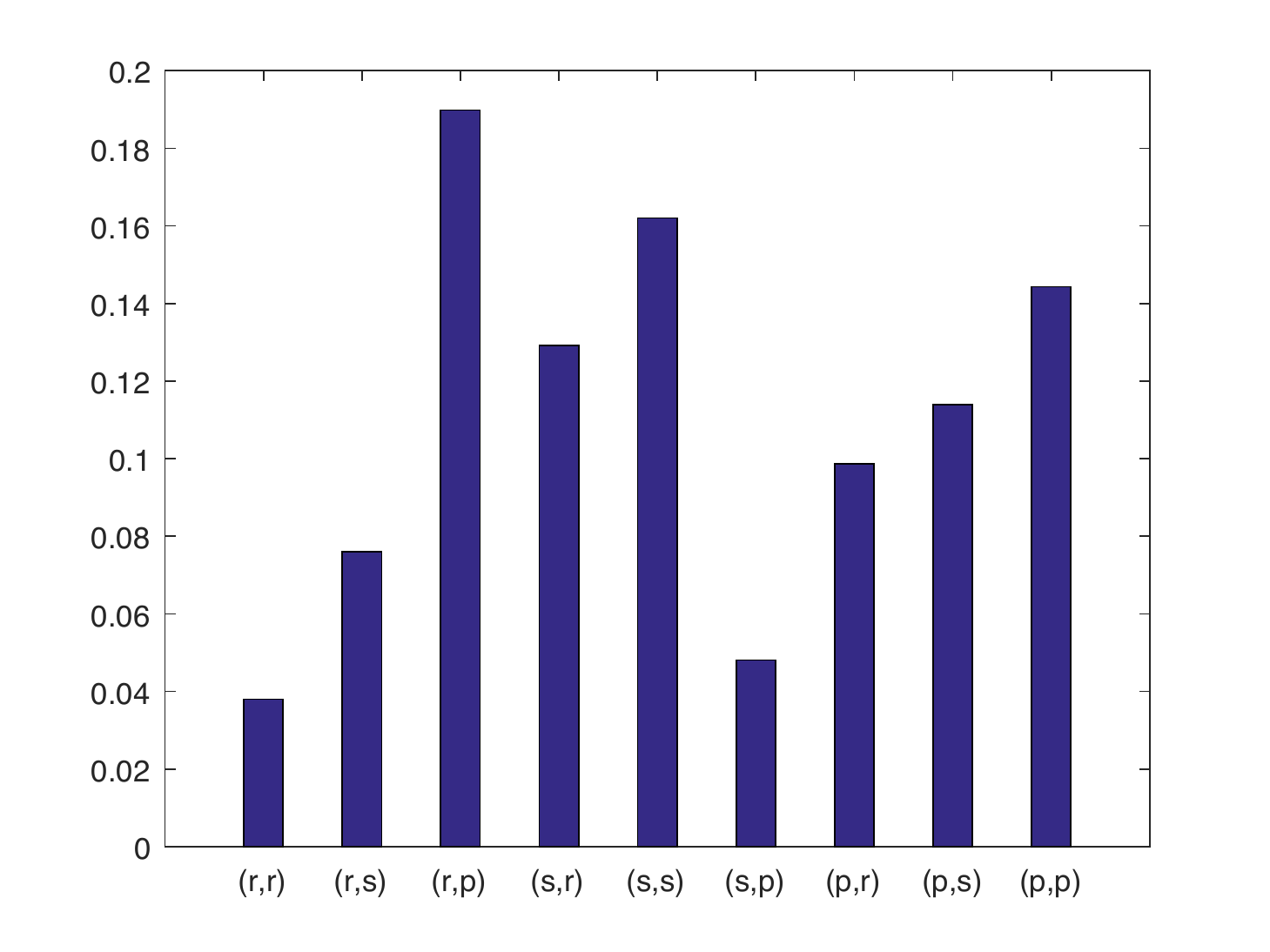}}
\caption{The invariant measure $\rho^*$ for Rock-Scissors-Paper with constraints}
\label{rspconstraint}
\end{figure}
From Figure \ref{rspconstraint}, we observe several properties that accord with modeling intuitions.
Firstly, player 1 is at disadvantage to player 2, since the chance of player 1 winning is less than that of player 2,
\begin{equation*}
\rho^*_{(r, s)}+\rho^*_{(p, r)}+\rho^*_{(s, p)}=0.2228<0.4329=\rho^*_{(s, r)}+\rho^*_{(r, p)}+\rho^*_{( p, s)}\ .
\end{equation*}
Secondly, we see that player 1 and 2's probabilities are not independent, meaning that they make decisions depending on each others' choices. 
Thirdly, from player 1's perspective, by assuming player 2 selected strategies uniformly, player 1 would choose Paper more frequently than Rock and Scissors due to the constraint. Thus in turn by taking advantage of this information, player 2 would have selected Paper (0 cost)  
or Scissors (-1 cost). This is reflected by Figure \ref{rspconstraint} that the top three states with highest probabilities are $(r,p),(s,s)$ and $(p,p)$.
%

\section{Conclusion}
We summary all features of the proposed dynamic framework:
First, the model incorporates players myopicity, uncertainty and greedy when making decisions;
Second, the model works for both potential and non-potential games. For potential games, the ranking of Nash equilibria given by the limit distribution coincides with the ranking given by the potential; For non-potential games, this ranking relates to the Morse decomposition and Conley-Markov matrix proposed in \cite{nature2012};
Last but not least, the FPE converges to Gibbs measure for potential games. The convergence is exponentially fast, whose rate is controlled by the relation between discrete entropy and {Fisher information} \cite{li-theory, Fisher}.

\end{document}